\documentclass[12pt]{amsart}
\usepackage{amsmath}
\usepackage{amssymb}

\input{epsf}
\input xy
\xyoption{all}

\newtheorem{fact}{Fact}[section]

\newtheorem{theorem}[fact]{Theorem}
\newtheorem{definition}[fact]{Definition}
\newtheorem{example}[fact]{Example}
\newtheorem{rremark}[fact]{Remark}

\newenvironment{remark}{\begin{rremark} \rm}{\end{rremark}}

\newtheorem{conjecture}[fact]{Conjecture}

\DeclareMathOperator\N{\mathbf N} \DeclareMathOperator\C{\mathbf C}
\DeclareMathOperator{\CC}{\mathcal C} \DeclareMathOperator\Z{\mathbf
Z} \DeclareMathOperator\codim{codim} \DeclareMathOperator\spa{span}
\DeclareMathOperator\rk{rank}\DeclareMathOperator\PP{\mathbf P}

\DeclareMathOperator\q{\mathbf q}\DeclareMathOperator\w{\mathbf w}
\DeclareMathOperator{\DD}{\mathcal D}
\DeclareMathOperator{\diag}{diag}
\DeclareMathOperator{\Gr}{Gr}
\DeclareMathOperator{\Hom}{Hom}

\title{Equivariant classes of matrix matroid varieties}

\author{L. M. Feh\'er, A. N\'emethi, R. Rim\'anyi}

\dedicatory{Dedicated to the memory of T. Brylawski (1944-2007)}

\begin{document}

\begin{abstract} Consider an integer associated with every subset of the set of columns of an $n\times k$ matrix. The collection of those matrices for which the rank of a union of columns is the predescribed integer for every subset, will be denoted by $X_{\CC}$. We study the equivariant cohomology class represented by the Zariski closure $Y_{\CC}$ of this set. We show that the coefficients of this class are solutions to problems in enumerative geometry, which are natural  generalization of the linear Gromov-Witten invariants of projective spaces. We also show how to calculate these classes and present their basic properties.
\end{abstract}

\maketitle

\section{Introduction}
\subsection{Matroid representation varieties} Consider an integer associated with every subset of the set of columns of an $n\times k$ matrix. The collection of those complex matrices for which the rank of a union of columns is the predescribed integer for every subset, will be denoted by $X_{\CC}$. Our main object of study in this paper is the Zariski closure $Y_{\CC}$ of this set. This is a version of matroid representation varieties. Other versions (eg. contained in Grassmannians, instead of the affine space of matrices) are also known, and they are closely related to ours. A dual point of view is considering the hyperplanes determined by the column vectors of the matrices. From this point of view $Y_{\CC}$ is the parameter space of certain hyperplane arrangements.

Matroid representation varieties are universal objects in algebraic geometry in the sense that any complication of varieties can be modeled on them.
The precise statement of this universality theorem is called Mn\"ev's theorem, see \cite{mnev}, \cite{RG_onMnev}, or a  recent account in \cite{vakil_onMnev}. Hence one does not hope that any reasonable question on these varieties has an easy answer. One manifestation of this phenomenon is the determination of the ideal of these varieties. In Section~\ref{sec:ideal} below we will explain with examples how the generators of the ideal encode projective geometry theorems.

The problem we will consider about matroid varieties is an enumerative geometry problem, a generalization of the linear Gromov-Witten invariants of projective spaces. Suppose a matroid variety is given, as above. Consider $k$ generic linear subspaces $V_i$ in $\C^n$. The question is, how many $n\times k$ matrices exist that belong to our matroid variety such that the $i$'th column vector is in $V_i$. For example, after projectivizing, we can ask the following question: given 8 generic straight lines and a generic point in the projective plane, how many Pappus configurations exist with 8 points of the Pappus configuration belonging to the 8 lines, and the 9'th point coinciding with the given point. The precise definitions, and the answer are given below. In the special case, when the matroid variety is the variety of rank $\leq 2$ matrices of size $n\times k$, this enumerative question is equivalent to the determination of $k$-point linear Gromov-Witten invariants in projective spaces. For general matroid varieties, however, no classical geometric or Gromov-Witten-type methods are known to compute the generalized Gromov-Witten invariants.

The nature of the matroid Gromov-Witten invariants in $\PP^2$ can be visualized by pictures. Some interactive presentations, created with the Interactive Geometry Software Cinderella \cite{cinderella} can be found at www.unc.edu/\~{}rimanyi/matroid\_{}show.

We will show in Theorem \ref{coeff_GW} that our matroid versions of linear Gromov-Witten invariants can be computed through the equivariant classes $[Y_{\CC}]$. These are cohomology classes that the varieties $Y_{\CC}$ represent in the $GL(n)\times GL(1)^k$-equivariant cohomology ring of the vector space of $n\times k$ matrices.

\subsection{Equivariant classes represented by invariant varieties of a representation}
Let the group $G$ act on the complex vector space $V$, and let $Y\subset V$ be an invariant variety of complex codimension~$c$. Then $Y$ represents a cohomology class $[Y]\in H^{2c}_G(V)$ in equivariant cohomology. There are various definitions and names for this class, eg. equivariant Poincar\'e dual, Thom polynomial, multidegree. We will call it the equivariant class of $Y$. Since $H^*_G(V)$ is naturally isomorphic to the ring $H^*(BG)$ of $G$-characteristic classes, the equivariant class $[Y]$ is simply a $G$-characteristic class of degree $2c$.

The equivariant class of the variety $Y$ encodes a lot of geometric information on $Y$, let us just allude to the effectiveness of Schubert calculus (the Giambelli formula is such an equivariant class) or the generalization involving classes of quiver loci. Other applications include the enumerative geometry results coming from Thom polynomials of singularities, see eg. \cite{kleiman} for a classic review or \cite{quad} for a recent addition.

Let us remark that the equivariant class of a matroid representation variety can be interpreted as a class of a quiver loci for the ``star quiver'' (based on a star shaped graph). However, the equivariant study of quiver representations is only well understood for quivers of Dynkin type ADE, see \cite{tegez}, \cite{ks} and references therein.

The usual tools to calculate equivariant classes represented by invariant subvarieties involve equivariant resolution, equivariant degeneration, or equivariant localization techniques. These techniques require the understanding of the ideal or the singularities of the variety in question.
For matroid varieties we lack this essential information.

Another main approach to calculate equivariant classes, effective for equivariant classes of contact singularities as well, is an interpolation method described in \cite{cr}. Below we will study an improvement of this interpolation method. In essence, we will describe certain constraints that a particular $[Y_{\CC}]$ must satisfy. Some of these constraints originate from the topological arguments of \cite{cr}, some others from the enumerative interpretation of some coefficients. Finally, a third set of constraints follow from the analogue of a stabilization property recently proved for contact singularities \cite{dstab}. Overwhelming experience shows that these three sets of constraints are sufficient to determine the equivariant classes $[Y_{\CC}]$, providing several enumerative applications. However, at the moment, we have no theorem claiming that for a particular matroid a certain set of constraints is sufficient.

In Section \ref{dstab} we show a certain stabilization property connecting the equivariant classes of matroid varieties in different dimensions $n$. As a corollary we prove a vanishing theorem on certain coefficients. In Section \ref{calculation} we outline the method of calculating these invariants.

\medskip

A particularly interesting question, subject to future study, is whether the matroid Gromov-Witten invariants can be organized as structure constants of an algebraic object with some kind of associativity property---mimicking the construction of the (big) quantum cohomology ring.

\subsection{Acknowledgement} The authors would like to express their gratitude to T. Brylawski, S.~Fomin, A. Hrask\'o, and B. Sturmfels for helpful discussions and remarks; and to D. Adalsteinson for letting us use his computer cluster for our calculations.

\section{Matrix matroid varieties}

We will denote the set of natural numbers $\{0,1,\ldots\}$ by $\N$,
and the set $\{1,2,\ldots,k\}$ by $[k]$. For a set $X$ let $2^X$
denote its power set, ie. the set of subsets of $X$.
We will identify the vector spaces $(\C^n)^k$ and $\C^{n\times
k}$ by the rule
\begin{equation}\label{identify}
(v_1,\ldots,v_k) \leftrightarrow
\begin{pmatrix} \begin{pmatrix} \ \\ v_1 \\ \ \end{pmatrix} \ldots
\begin{pmatrix} \ \\ v_k \\ \ \end{pmatrix} \end{pmatrix}
\end{equation}
Elements in $(\C^n)^k$ will be referred to as (ordered) {\em vector
configurations} in $\C^n$. For a matrix $M\in \C^{n\times k}$ and subsets $U \subset [n]$,
$V\subset [k]$, let $M^U_V$ denote the submatrix consisting of the
$(i,j)$-entries of $M$ for $i\in U$, $j \in V$. Let $M_V=M_V^{[n]}$.

\medskip

The vector configuration $\CC=(v_1,v_2,\ldots,v_k) \in \left( \C^n
\right)^k$ defines the rank function $r_{\CC}:2^{[k]} \to \N$,
$$r_{\CC}(V)=\dim \spa \{v_i\}_{i\in V}.$$

\begin{definition}
For a configuration $\CC$ we define
$$X_{\CC}=\{M \in \C^{n \times k}: \rk (M_V) = r_{\CC}(V) \text{\ for all\ } V\subset [k]\}.$$
The Zariski closure $\overline{X_{\CC}}\subset \C^{n\times k}$ will
be called the {\em matrix matroid variety} associated with $\CC$,
and will be denoted by $Y_{\CC}$.
\end{definition}

If we identify $n\times k$ matrices with $k$-tuples of $n$-vectors
as in (\ref{identify}), then $X_{\CC}$ consists of those
configurations whose rank function is the same as that of $\CC$. For
example, $\CC$ itself belongs to $X_{\CC}$. The matrix matroid
variety $Y_{\CC}$ consists of those configurations that are limits
(degenerations) of elements in $X_{\CC}$.

Observe that $X_{\CC}$ and $Y_{\CC}$ do not change if we re-scale,
ie. multiply, any vector $v_i$ in $\CC$ by any non-zero complex
number. Hence $X_{\CC}$ and $Y_{\CC}$ are determined by the list of
points $P_i:=[v_i]\in \PP^{n-1}$ for  $v_i\not=0$, and the list of
those $v_i$ which are 0. By abusing language, such a list will also
called a configuration.

\begin{example} \label{ex: concrete321}
\rm Let $n=2, k=6$, and consider the following configuration
$$P_1=P_2=P_3=(0:1)\in \PP^1, \qquad P_4=P_5=(1:1)\in \PP^1, \qquad v_6=0.$$
This configuration is illustrated in Figure \ref{figure1}.
\begin{figure}
\epsffile{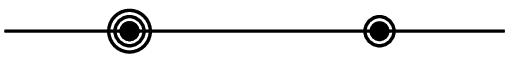} \put(-175,0){$P_1, P_2, P_3$}
\put(-90,0){$P_4,P_5$} \put(-15,25){$\emptyset$} \put(-15,0){$P_6$}
\caption{ } \label{figure1}
\end{figure}
Matrices in $X_{\CC}$ are those $2\times 6$ matrices whose
\begin{itemize}
\item{} first three columns are proportional {\em non-zero} vectors,
\item{} the forth and fifth columns are proportional {\em non-zero}
vectors,
\item{} {\em the first and the forth columns are non-proportional},
\item{} the sixth column is the zero vector.
\end{itemize}
It is true that $Y_{\CC}$ consists of matrices satisfying the
``closed'' conditions above, but not necessarily the ``open'' ones.
That is, $Y_{\CC}$ consists of matrices whose first three columns
are proportional, forth and fifth columns are proportional, and
sixth column is 0. However, the easy procedure of
dropping the open conditions will {\em not} specify $Y_{\CC}$ in general.
\end{example}

\begin{example} \label{ex:matrix schubert varieties}
 \rm {\em Matrix Schubert varieties.} Consider a
complete flag of linear spaces
$$L^0 \subset L^1 \subset \ldots \subset L^{n-1}\subset L^n$$
in $\C^n$. Let $\ell=(l_0,l_1,\ldots,l_n)\in \N^{n+1}$ with $\sum
l_i=k$. Choose $l_i$ generic points $v^{(i)}_1,\ldots,v^{(i)}_{l_i}$
in $L^i$. The matrix matroid variety corresponding to the
configuration
$$\CC_{\ell}=\left(v^{(0)}_1,\ldots,v^{(0)}_{l_0},v^{(1)}_1,\ldots,v^{(1)}_{l_1}
,\ldots,v^{(n)}_1,\ldots,v^{(n)}_{l_n}\right)$$ is studied in the
papers \cite[Sect.5]{ss}, \cite{knutson-miller-ss}, and is called
the matrix Schubert variety corresponding to Grassmannian
permutations.
\end{example}

For the problems to be considered later in this paper, matrix
Schubert varieties will be the simple case. Products of matrix
Schubert varieties will also be considered simple. Note, that
Example \ref{ex: concrete321} is such a product of matrix Schubert
varieties (after identifying $\C^{n\times k_1}\times \C^{n \times
k_2}$ with $\C^{n\times (k_1+k_2)}$), namely,
\begin{equation}\label{decomposition321}
Y_{\CC}=Y_{\CC_{(0,3,0)}} \times Y_{\CC_{(1,2,0)}}=Y_{\CC_{(0,3,0)}}
\times Y_{\CC_{(0,2,0)}}\times Y_{\CC_{(1,0,0)}}.\end{equation}

Examples of matrix matroid varieties which are not products of matrix
Schubert varieties will be given below.

\begin{remark} Other candidate names for matrix matroid varieties
would be  ``matroid variety'' or ``matroid representation variety''.
We chose the name ``matrix matroid variety'', because of the analogy
with matrix Schubert varieties.
\end{remark}

\section{The ideal of matrix matroid varieties}\label{sec:ideal}

The rank of a matrix is $r$ if its
$(r+1)\times (r+1)$ minors vanish, and at least one $r \times r$
minor does not vanish. Hence the algebraic description of $X_{\CC}
\subset \C^{n\times k}$ is a collection of equations (several minors
vanish), together with a collection of conditions expressing that
certain polynomials (some other minors) do not vanish together
(cf.~the open and closed conditions of Example \ref{ex:
concrete321}):
\begin{align*} X_{\CC}=\{ M=(m_{i,j}) \in \C^{n \times k} : \ &
p_u(m_{i,j})=0\ \text{for}\
u=1,2,\ldots; \\
& q^{(v)}_1(m_{i,j}),\ldots, q^{(v)}_{w_v}(m_{i,j})\  \text{do not
vanish together for}\ v=1,2,\ldots \}.\end{align*} It follows that
\begin{equation} \label{naive}
Y_{\CC} \subset \{M=(m_{i,j}) \in \C^{n \times k} :p_u(m_{i,j})=0\
\text{for}\ u=1,2,\ldots \}.
\end{equation}
Although it is tempting to think that we have equality in formula
(\ref{naive}), in general, this is not the case. First we give an
intuitive reason for this.

\subsection{Motivation: the Menelaus configuration}\label{menelaus_section}

 Consider the Menelaus
configuration $\CC_{M}$ of Figure \ref{menelaus ceva_figure}, with $n=3$,
$k=6$.

\begin{figure} \epsfysize=1.5in
\epsffile{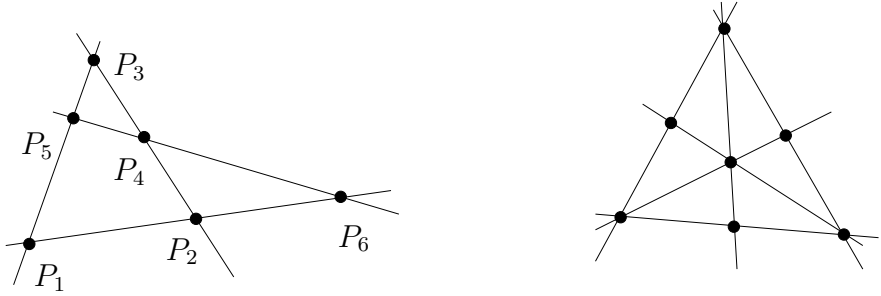}
\put(-320,0){$P_1$}
\put(-270,10){$P_2$}
\put(-290,80){$P_3$}
\put(-290,40){$P_4$}
\put(-325,50){$P_5$}
\put(-205,15){$P_6$}
\caption{The Menelaus configuration $\CC_M$, and the Ceva configuration $\CC_C$}
\label{menelaus ceva_figure}
\end{figure}

The
equations $p_u$ of formula (\ref{naive}) are the four $3 \times 3$
minors of the $3\times 6$ matrix $(m_{i,j})$ corresponding to the
following triples of column-indices: 126, 135, 234, 456. The right
hand side of formula (\ref{naive}) hence contains all $3\times 6$
matrices for which these four minors vanish. We claim that there is
a matrix for which these minors vanish, but it is {\em not} in
$Y_{\CC_M}$, ie. it is not a limit of matrices from $X_{\CC_M}$.
Indeed, consider an affine chart of $\PP^2$, a line $l$ in it, and
the configuration $\CC'$ of six generic points on $l$. If this
configuration was in the closure of $X_{\CC_M}$ then there would be
a family of configurations belonging to $X_{\CC_M}$, all in the
affine chart, converging to $\CC'$. For all these configurations
Menelaus' theorem \cite{menelaus_original} holds, which we recall
now.

\begin{theorem}\label{menelaus_affine}{\em [Complex affine version of Menelaus's
theorem]}
Consider the configuration $\CC_M$ of points in $\C^2$. Choose an
identification of the $P_1 P_3 P_5$ line with $\C$. Observe that the
complex number $( P_5-P_1 )/(P_3-P_5)$  does not depend on the
choice; denote this ratio by $P_1P_5/P_5P_3$. Then (using similar
notations for the other straight lines) we have
\begin{equation}\label{menelaus identity}\frac{P_1P_5}{P_5P_3} \cdot \frac{P_3P_4}{P_4P_2} \cdot \frac{P_2P_6}{P_6P_1} = -1.\end{equation}
\end{theorem}

Our reasoning is finished by observing that for $\CC'$ the Menelaus
identity (\ref{menelaus identity}) does not hold; this proves that
for $\CC_M$ we have strict $\subset$ in Formula (\ref{naive}).

One may wonder if there is a complex projective version of
Menelaus's theorem, which would eliminate the need for the affine
chart in the geometric proof above. The answer is given in the next
section.

\bigskip

What we learned from the Menelaus example is that
\begin{itemize}
\item{} the ``naive'' equations $p_u$ are not enough to cut out
$Y_{\CC}$ even set-theoretically from $\C^{n\times k}$;
\item{} the extra equations needed (besides the naive ones) encode
the not-so-obvious geometric theorems of the configuration.
\end{itemize}

\subsection{The ideal of $X_{\CC}$, examples.} \label{examples:ideal}

Let $I_{\CC}$ denote the ideal of the variety $Y_{\CC}$, ie. the
homogeneous ideal of polynomials vanishing on $Y_{\CC}$.

\begin{example}\rm In Schubert calculus the following statement is well
known: For the matrix Schubert variety $\CC_{\ell}$ of Example
\ref{ex:matrix schubert varieties} the ``naive'' equations generate
$I_{\CC_{\ell}}$:
$$I_{\CC_{\ell}}=\left(\det \left(M^{\{i_1,\ldots,i_s\}}_{\{j_1,\ldots,j_s\}}\right)=0\right)$$
for $i_1<\ldots< i_s, j_1<\ldots<j_s, j_s\leq
l_0+\ldots+l_{s-1},{s=1,\ldots,n}$.
\end{example}

Now consider the Menelaus configuration of Section
\ref{menelaus_section}, and consider the variety corresponding to
the naive equations
$$Y_{naive}=\{M \in \C^{3\times 6}\ :\ \det M_{126}=0, \det M_{135}=0, \det M_{234}=0, \det M_{456}=0\}.$$
Computer algebra packages \cite{singular} can be used to find that this variety is
the union of two irreducible varieties
\begin{equation}\label{menelaus_decomposition}
Y_{naive}=Y_{\CC_M} \cup \{M \in\C^{3\times 6}\ :\
\rk(M)\leq 2\}.\end{equation} This decomposition sheds light on the
intuitive reasoning of Section \ref{menelaus_section}. As a
byproduct, the computer algebra package finds generators of the
ideal $I_{\CC_M}$. It turns out that $I_{\CC_M}$ can be minimally
generated by polynomials of degrees
$$3,3,3,3,\ 5,5,5,\ {\underbrace{{6,\ldots,6}}_{12}}.$$
The four degree 3 polynomials can be chosen to be the four naive
equations. As a consequence, the collection of the other equations
can be considered as the {\em extra}, non-trivial complex projective
identities holding for Menelaus configurations. We might as well
call this set of polynomials the ``complex projective Menelaus'
theorem''.

For completeness let us show how to generate degree 6 and degree 5
polynomials in $I_{\CC_M}$ knowing only the usual version of
Menelaus' theorem (Theorem \ref{menelaus_affine}). In the projective
plane $(x,y,z)$ we can choose the $y$ coordinate to be at infinity.
In the remaining affine plane we can identify ratios of complex
numbers by appropriate projections. Hence, for example, from
(\ref{menelaus identity}) we can obtain
\begin{equation}\label{elso}
\frac{{x_5}/{y_5}-{x_1}/{y_1}}{{x_3}/{y_3}-{x_5}/{y_5}} \cdot
\frac{{z_4}/{y_4}-{z_3}/{y_3}}{{z_2}/{y_2}-{z_4}/{y_4}}
\cdot\frac{{z_6}/{y_6}-{z_2}/{y_2}}{{z_1}/{y_1}-{z_6}/{y_6}}=-1.
\end{equation}
Rearranging this equality we obtain a degree 6 polynomial. Making
other choices we may obtain several other degree 6 polynomials.
Getting degree 5 ones is more delicate. Consider the degree 6
polynomial obtained from (\ref{elso}), and the ones obtained from
the next three similar equalities
\begin{equation*}
\frac{{z_5}/{y_5}-{z_1}/{y_1}}{{z_3}/{y_3}-{z_5}/{y_5}} \cdot
\frac{{z_4}/{y_4}-{z_3}/{y_3}}{{z_2}/{y_2}-{z_4}/{y_4}}
\cdot\frac{{x_6}/{y_6}-{x_2}/{y_2}}{{x_1}/{y_1}-{x_6}/{y_6}}=-1,
\end{equation*}
$$\frac{{x_1}/{y_1}-{x_2}/{y_2}}{{x_2}/{y_2}-{x_6}/{y_6}} \cdot
\frac{{z_6}/{y_6}-{z_4}/{y_4}}{{z_4}/{y_4}-{z_5}/{y_5}}
\cdot\frac{{z_5}/{y_5}-{z_3}/{y_3}}{{z_3}/{y_3}-{z_1}/{y_1}}=-1,$$
$$\frac{{z_1}/{y_1}-{z_2}/{y_2}}{{z_2}/{y_2}-{z_6}/{y_6}} \cdot
\frac{{z_6}/{y_6}-{z_4}/{y_4}}{{z_4}/{y_4}-{z_5}/{y_5}}
\cdot\frac{{x_5}/{y_5}-{x_3}/{y_3}}{{x_3}/{y_3}-{x_1}/{y_1}}=-1.$$
It turns out that the sum of these four degree 6 polynomials is
$y_4$ times
\begin{equation}\label{deg5_generator} -x_5
y_1z_3z_6y_2+x_5 y_1 z_3 z_2 y_6+x_3 y_5 z_2 z_1 y_6-x_5 y_3 z_2 z_1
y_6+ z_1 y_5 z_3 x_6 y_2-z_1 y_5 z_3 x_2 y_6-\end{equation}
$$z_3
y_5 z_2 x_6 y_1+z_5 y_3 z_2 x_6 y_1+ x_2 y_1 z_6 z_3 y_5-x_2 y_1 z_6
z_5 y_3-x_6 y_2 z_5 z_1 y_3+$$
$$x_2 y_6 z_5 z_1 y_3- z_1 y_2 z_6 x_3 y_5+z_1 y_2 z_6 x_5 y_3+z_6
y_2 z_5 x_3 y_1-z_2 y_6 z_5 x_3 y_1.$$ This latter is one of the
degree 5 generators of $I_{\CC_M}$. The other two can be obtained by
similar calculations, or appropriate changes of variables in
(\ref{deg5_generator}).

\bigskip

For an arbitrary configuration $\CC$ determining $I_{\CC}$ seems to be a
hopelessly difficult problem. One might start with the ideal generated by
the naive equations, and try to get rid of the ``fake'' components
(just like the determinantal variety in
(\ref{menelaus_decomposition})) by primary decomposition or by
dividing (or saturating) with ideals of extra components. In
practice, none of these strategies is feasible in reasonable time
for configurations even a little more complicated than the Menelaus
configuration.

\smallskip

Below we will study another invariant of matrix matroid varieties,
namely their equivariant classes, which will be much better
computable than $I_{\CC}$, and which can answer various questions
about these varieties without determining their ideals.

\section{Equivariant classes of matrix matroid varieties.}

We will work in the complex algebraic category; cohomology will
be meant with integer coefficients; and $GL(n)$ will denote the general linear group $GL(n,\C)$.

\subsection{Equivariant classes in general}

If $Y$ is a complex codimension $c$ subvariety in a compact complex manifold
$M$, then $Y$ represents a cohomology class $[Y]$ in $H^{2c}(M)$.
The following, equivariant version of this notion is more delicate
to define, see e.g. \cite{cr}, \cite[8.5]{miller-sturmfels}, \cite{fultonnotes}.

Let the group $G$ act on the complex vector space $V$, and let $Y\subset V$
be an invariant variety of complex codimension~$c$. Then $Y$
represents a cohomology class $[Y]\in H^{2c}_G(V)$ in equivariant
cohomology. Since $H^*_G(V)$ is naturally isomorphic to the ring
$H^*(BG)$ of $G$-characteristic classes, the equivariant class $[Y]$
is simply a $G$-characteristic class of degree $2c$.

\subsection{Equivariant classes of matrix matroid varieties}
\label{equivariant_class_section}

Let $D(k)$ be the group of diagonal matrices of size $k$. Consider
the action of $G_{n,k}=GL(n) \times D(k)$ on the vector space
$\C^{n\times k}$ of $n\times k$ matrices by
$$(A,B) \cdot M = A M B^{-1}, \qquad\qquad A\in GL(n), B \in D(k), M\in \C^{n\times k}.$$
Viewing elements of $\C^{n\times k}$ as vector configurations as in
(\ref{identify}), the action of $(A,B)\in G_{n,k}$ reparametrizes
$\C^n$ (the action of $A$) and rescales the vectors one by one (the
action of $B$). Therefore, the spaces $X_{\CC}$ and hence the
varieties $Y_{\CC}$ are $G_{n,k}$-invariant.

In the rest of the paper the main concept of interest will be the
equivariant class \begin{equation}\label{ring}[Y_{\CC}]\in
H^*_{G_{n,k}}(\C^{n\times
k})=H^*(BG_{n,k})=\Z[c_1,\ldots,c_n,d_1,\ldots,d_k],\end{equation}
where $c_i$ are the Chern classes of $GL(n)$, and $d_i$ are the
first Chern classes of the $GL(1)$ components of $D^k=GL(1)^k$. We
have $\deg c_i=2i, \deg d_i =2$.

\subsection{Examples}

In Sections \ref{section:GW}--\ref{hierarchy} we will show how to calculate
the classes $[Y_{\CC}]$, and discuss their geometric meaning.  Before that,
however, we show some examples.

\begin{example}\label{321tp} \rm
 Consider the configuration $\CC$ of Example \ref{ex:
concrete321}. We have
$$[Y_{\CC}]=\left(d_1d_2+d_1d_3+d_2d_3-c_1(d_1+d_2+d_3)+c_1^2-c_2\right)
(c_1-d_4-d_5)(d_6^2-c_1d_6+c_2).$$
\end{example}

\bigskip

Let us now consider matrix Schubert varieties $\CC_\ell$ of Example
\ref{ex:matrix schubert varieties}. That is, we have
$\ell=(l_0,\ldots,l_n)$, $\sum l_i=k$. We may assume without loss of
generality that there is an $r$ such that $l_1, \ldots, l_r$ are all
non-zero, while $l_{r+1}=\ldots=l_n=0$. (Indeed, observe for example
that $Y_{\CC_{2,0,2}}=Y_{\CC_{2,1,1}}$, by changing the complete
flag.) Define
$$\mu_i= \begin{cases} \sum_{j=0}^{i-1} l_j +1, &  i\leq r \\
k+i-r & i> r,\end{cases}\qquad\qquad \lambda_{n+1-i}=\mu_i-i,
\qquad\qquad \text{for}\ i=1,\ldots,n.$$ Let $\beta^{(i)}_j$ be
degree $j$ polynomials in the ring (\ref{ring}), defined by
$$1+\beta_1^{(i)}t+\beta_2^{(i)}t^2+\ldots = \frac{ \prod_{j<\mu_i}
(1+d_j t)}{1+c_1t+\ldots+c_nt^n}.$$

\begin{theorem}\label{schur_tp}
Using the notation above, the matrix Schubert variety $\CC_\ell
\subset \C^{n\times k}$ has complex codimension
$|\lambda|=\sum\lambda_i$, and we have
\begin{equation}\label{schur}
[Y_{\CC_\ell}]= (-1)^{|\lambda|}\det \left(
\beta^{(n+1-i)}_{\lambda_i+j-i}
\right)_{i,j=1,\ldots,n}.\end{equation}
\end{theorem}

\begin{proof}
This is, in fact, not a new theorem. Observe, that $Y_{\CC_\ell}$ is
not only invariant under the action of $GL(n) \times D(k)$, but
under the same action of $GL(n) \times B(k)$, where $B(k)$ is the
Borel group of upper triangular $k \times k$ matrices associated with the complete flag. The varieties
$Y_{\CC_\ell}$ are, in fact, the orbit closures of this extended
action. The equivariant classes of these orbit closures are
calculated in \cite[Thm. 5.1]{ss} to be the double Schur polynomials
of (\ref{schur}), see also \cite{knutson-miller-ss}. Since the
inclusion $D(k)\subset B(k)$ is a homotopy equivalence, the
$G_{n,k}$-equivariant classes are the same as the $GL(n)\times
B(k)$-equivariant classes.
\end{proof}

Equivariant classes of products of varieties multiply in the obvious
sense. For instance, the result of Example \ref{321tp} can be
recovered from the factorization (\ref{decomposition321}) and the
application of Theorem~\ref{schur_tp} to the three factors. Namely, we have
$$[Y_{\CC_{0,3,0}}]=\det\begin{pmatrix}
\frac{\prod_{i=1}^3 (1+d_it)}{1+c_1t+c_2t^2}|_{2} &
\frac{\prod_{i=1}^3
(1+d_it)}{1+c_1t+c_2t^2}|_{3} \\
0 & \frac{\prod_{i=1}^0 (1+d_it)}{1+c_1t+c_2t^2}|_{0}
\end{pmatrix}=d_1d_2+d_1d_3+d_2d_3-c_1(d_1+d_2+d_3)+c_1^2-c_2.$$
Here, and later, $f(t)|_i$ means the $i$'th coefficient of the
Taylor series $f(t)$ in the formal variable~$t$. Similarly,
$$[Y_{\CC_{0,2,0}}]=
-(d_4+d_5-c_1),\ \qquad [Y_{\CC_{(1,0,0)}}]=d_6^2-c_1d_6+c_2,
$$ after appropriate shifting
of indices.

\bigskip

It is rather difficult to present equivariant classes of matrix matroid
varieties which are not products of matrix Schubert varieties. For
example the class $[Y_{\CC_M}]$ for the Menelaus configuration of
Section \ref{menelaus_section} is a degree 8 polynomial with 173
terms (in $c$-$d$-monomials). To indicate how it looks we show this
polynomial after we substitute 0 for all the $d_i$ variables: $[Y_{\CC}]^*=[Y_{\CC}]_{d_i=0\forall i}$.
Observe that the $[Y_{\CC}]^*$ class is the $GL(n)$ equivariant class represented by $Y_{\CC}$. We have

\begin{itemize}
\item{} $[Y_{\CC_M}]^*= 3c_1^2c_2-2c_1c_3-c_2^2=3\Delta_{(211)}+2\Delta_{(22)}+3\Delta_{(31)}.$
\end{itemize}

Here the Schur polynomials $\Delta_{\lambda}$ corresponding to a partition $(\lambda_1\geq \lambda_2\geq\ldots\geq \lambda_r)$ are defined
by $\Delta_{\lambda}=\det(c_{\lambda_i+j-i})_{i,j=1,\ldots, r}$. The significance of the Schur basis is presented in Sections \ref{section:proof}, \ref{dstab}. The expression in Theorem \ref{schur_tp} can also be interpreted as a Schur polynomial.
Here is a list of similar specializations of $[Y_{\CC}]$ for the Ceva, Pappus, and Desargues configurations.

\begin{figure}
\epsffile{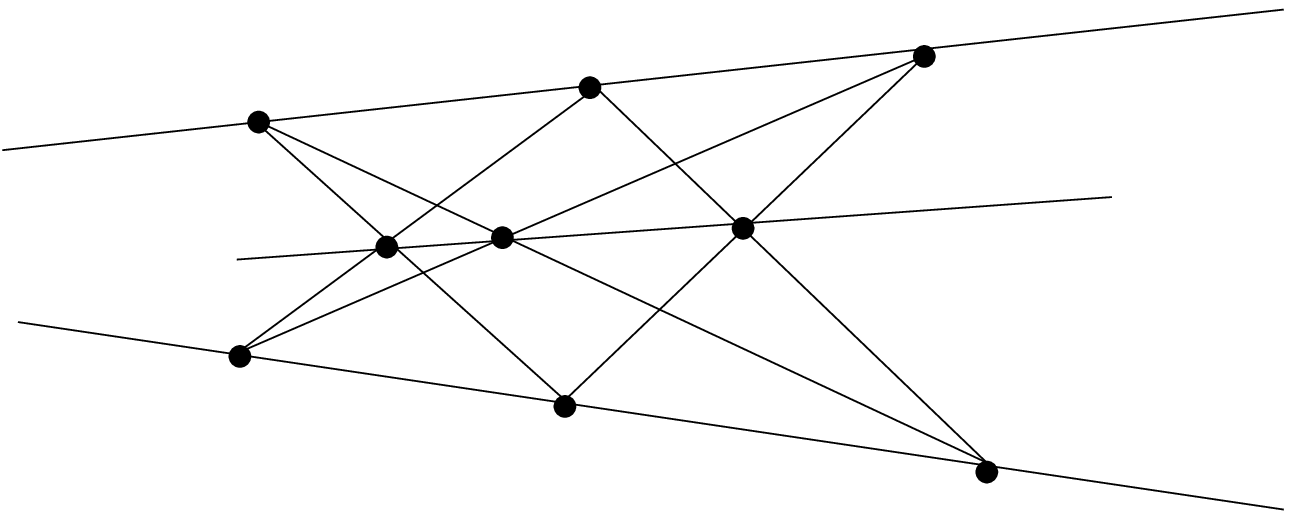} \put(-300,130){$P_1$} \put(-200,140){$P_2$}
\put(-100,150){$P_3$} \put(-285,80){$P_7$} \put(-230,90){$P_8$}
\put(-140,70){$P_9$} \put(-300,30){$P_4$} \put(-200,13){$P_5$}
\put(-100,-5){$P_6$}\caption{The Pappus configuration $\CC_P$}
\label{pappus_pic}
\end{figure}

\begin{itemize}
\item{} $[Y_{\CC_C}]^*=6\Delta_{(2211)}+4\Delta_{(222)}+3\Delta_{(3111)}+8\Delta_{(321)}+\Delta_{(33)}$;
\item{} $[Y_{\CC_P}]^*=11\Delta_{(221111)}+16\Delta_{(22211)}+8\Delta_{(2222)}+12\Delta_{(311111)}+28\Delta_{(32111)}+28\Delta_{(3221)}+17\Delta_{(3311)}+15\Delta_{(332)}$;
\item{} $[Y_{\CC_D}]^*=15\Delta_{(222111)}+20\Delta_{(22221)}+20\Delta_{(321111)}+50\Delta_{(32211)}+30\Delta_{(3222)}+30\Delta_{(33111)}+45\Delta_{(3321)}+10\Delta_{(333)}$.
\end{itemize}

The equivariant classes are rather meaningless formulas unless we
find geometric applications. We finish this section with a rather
simple application; more delicate geometric meaning will be
discussed in Sections \ref{section:GW} and \ref{section:proof}.

\subsection{The degree of $\PP(Y_{\CC})$}
Suppose the group $G$ acts on the vector space $V$, and $Y$ is an
invariant cone. Then the degree of the projective variety $\PP(Y)$
can be recovered from the equivariant class $[Y]\in H^*_G(V)$ by the
following procedure.

Let $T^n$ be a maximal torus of $G$ with corresponding Chern roots
$\alpha_i$. If $w_1,\ldots,w_n, w$ are integers with the property
that for any $z\in \C$, $|z|=1$ we have
$$(z^{w_1},\ldots,z^{w_n}) \cdot v = z^w v, \qquad (z^{w_1},\ldots,z^{w_n})\in T^n, v\in V,$$
then
$$\deg \PP(Y) = [Y] (\alpha_i= \frac{w_i}{w}).$$
On the right hand side we have the equivariant class, with the
number $w_i/w$ substituted into the Chern root corresponding to the
$i$'th factor of $T^n$. This theorem seems to be a folklore
statement, a recent proof is given eg. in \cite[6.4]{fnr_form}.

\smallskip

For matrix matroid varieties we have two natural choices for the
substitution. Either we substitute
$$c_i=\binom{n}{i}, d_i=0,\qquad\text{or}\qquad c_i=0, d_i=-1.$$
Observe that the first substitutions can be carried out for the
specialized classes above ($d_i=0 \forall i$), hence the following
theorem can be checked from the classes presented above.

\begin{theorem} For the Menelaus, Ceva, Pappus, and Desargues
configurations (see Figures \ref{menelaus ceva_figure},
\ref{pappus_pic}, \ref{desargues_pic}) we have
$$\deg \PP(Y_{\CC_M}) = 66,\qquad \deg \PP(Y_{\CC_C})=297,\qquad
\deg \PP(Y_{\CC_P}) =2943 ,\qquad \deg \PP(Y_{\CC_D})=4680.$$
\end{theorem}

\begin{remark} \rm
The degree of $\PP(Y_{\CC_M})$ can also be recovered from the
decomposition in (\ref{menelaus_decomposition}). Indeed, in this
decomposition all three varieties are 4 codimensional; $Y_{naive}$
has degree $3^4$ because of B\'ezout's theorem; the determinantal
variety $\{M\in \C^{3\times 6}|\rk M\leq 2\}$ has degree $\binom{6}{4}$ (see \cite[14.4.14]{fulton_Int_Th}).
Hence $\deg Y_{\CC_M}=81-15=66$. The same argument shows that
$$[Y_{\CC_M}]=(c_1-d_1-d_2-d_6)(c_1-d_1-d_3-d_5)(c_1-d_2-d_3-d_4)(c_1-d_4-d_5-d_6)-\frac{\prod_{j=1}^6(1+d_jt)}{1+c_1t+c_2t^2+c_3t^3}|_4.$$
For the other varieties in the Theorem we know no other way of determining
the degree, but to calculate the equivariant class as in Section \ref{calculation}, then carry out the described substitution.
\end{remark}

\begin{figure}\epsfysize=2.5in \epsfxsize=3.7in
\epsffile{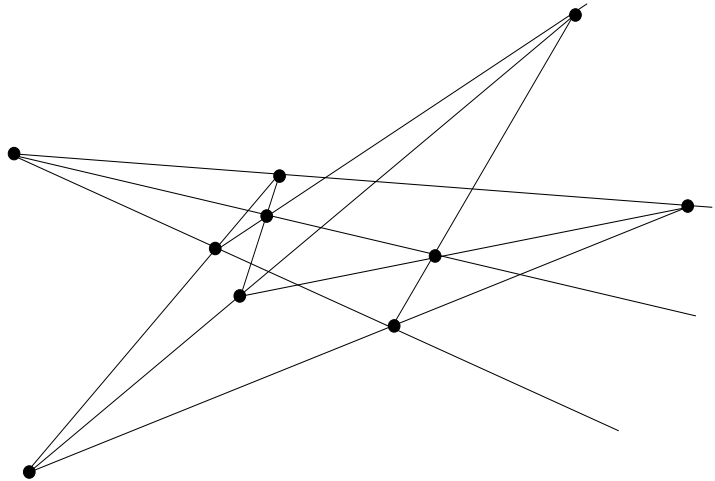} \caption{The Desargues configuration
$\CC_D$} \label{desargues_pic}
\end{figure}

\section{Matroid versions of linear Gromov-Witten
invariants}\label{section:GW}

In the rest of the paper, for simplicity, we will assume that $\CC$ is a configuration of $k$ {\em non-zero} vectors in
$\C^n$, and that $Y_{\CC}$ is pure dimensional.

For nonnegative integers $q_1,\ldots,q_k$ with $\sum q_i = \codim_{\C}
(Y_{\CC}\subset \C^{n\times k})$ we define
$$N(\CC;q_1,\ldots,q_k)=\#\{([v_1],\ldots,[v_k])\in (\PP \C^n)^k : (v_1,\ldots,v_k)\in Y_{\CC}, v_i \in
V_i\},$$ where  $V_1, \ldots, V_k$ is a generic collection of linear
spaces with $\dim V_i = q_i +1$.

More generally, instead of generic linear spaces $V_i$, we could
have considered varieties of different dimensions and degree. These generalized enumerative problems can be reduced
to the linear version above.

\begin{example} \label{steiner_construction}\rm
 Consider the configuration $\CC$ in Figure
\ref{steiner_pic} (a).
\begin{figure}\epsfysize=1.7in \epsfxsize=6in
\epsffile{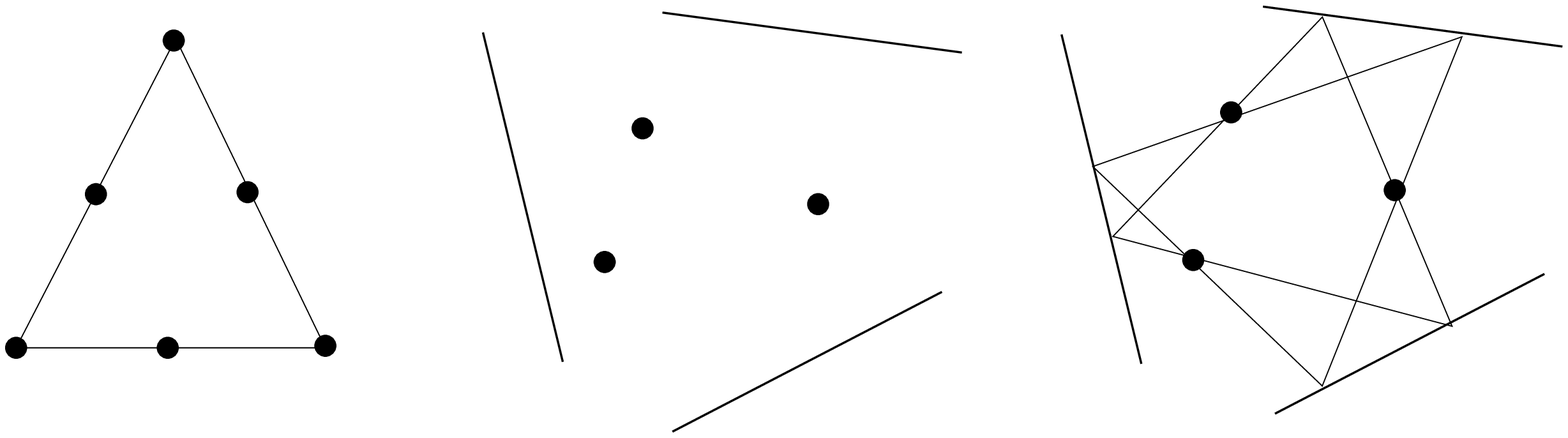}\put(-450,10){$P_1$}\put(-390,10){$P_6$}\put(-340,10){$P_2$}\put(-390,130){$P_3$}
\put(-430,70){$P_5$}
\put(-360,70){$P_4$}\put(-320,50){$\PP(V_1)$}\put(-205,5){$\PP(V_2)$}\put(-200,125){$\PP(V_3)$}
\put(-270,33){$\PP(V_6)$}
\put(-203,70){$\PP(V_4)$}\put(-265,95){$\PP(V_5)$}\caption{(a)
$\CC$, (b) the enumerative problem, (c) the solution $=2$}
\label{steiner_pic}
\end{figure}
The number $N(\CC;1,1,1,0,0,0)$ is the number of solutions to the
following problem: given 3 points and 3 straight lines (generically)
on the plane $\PP^2$ (Figure \ref{steiner_pic} (b)). How many
triangles exist, whose vertices are on the straight lines, and whose
sides pass through the given points? The solution is 2 (Figure
\ref{steiner_pic} (c)) due to the following well know argument:
Choosing a point $X$ on $\PP V_1$ we can project it through $\PP
V_6$ to $\PP V_2$, then project further through $\PP V_4$ to $\PP
V_3$, then further through $\PP V_5$ back to $\PP V_1$, obtaining a
point $X'$. The map $X \mapsto X'$ is a projective transformation of
the projective line $\PP V_1$, whose number of fixed points is the
question. Since projective transformations have the form $x \mapsto
(ax+b)/(cx+d)$ (in affine coordinate), the number of fixed points
is 2. (The construction of the fixed points using a compass and
straightedge is the famous Steiner construction, see e.g.
\cite[Ad.3]{geo_book}.)
\end{example}

\begin{example} \rm \label{mene_enum_example}
Consider the number $N(\CC_{M};1,1,1,1,0,0)$ for the Menelaus
configuration of Figure \ref{menelaus ceva_figure}. One may try to follow
the argument of Example \ref{steiner_construction}: choose a point
$X$ on $\PP V_1$, projecting it through $\PP V_6$ to $\PP V_2$, then
further project through the intersection of $\PP V_4$ and the line
$\overline{\PP V_5 ,\PP V_6}$ to $\PP V_3$, then even further
through the point $\PP V_5$, back to $\PP V_1$, obtaining $X'$. The
transformation $X \mapsto X'$ is a projective transformation of $\PP
V_1$, hence it has 2 fixed points, suggesting that
$N(\CC_{M};1,1,1,1,0,0)=2$. However, this is wrong, the correct
number is $N(\CC_{M};1,1,1,1,0,0)=1$. One of the two fixed points of
the transformation $X\mapsto X'$ is on the ``other'' component of
$Y_{naive}$ in (\ref{menelaus_decomposition}), not on $Y_{\CC_M}$.
Geometrically, one of the fixed points of the transformation
corresponds to all the points lying on the line of $\PP V_5$ and
$\PP V_6$, which configuration does not belong to $Y_{\CC_M}$. There
are configurations (e.g. the Ceva configuration) for which some
``extra'' components in $Y_{naive}\setminus Y_{\CC}$ have bigger
dimensions then $Y_{\CC}$. For these, arguments similar to that in
Example \ref{steiner_construction} suggest the incorrect
$N(\CC;q_1,\ldots,q_k)=\infty$.
\end{example}

If the ideal $I_{\CC}$ is known, determining the numbers
$N(\CC;q_1,\ldots,q_k)$ reduces to algebraic calculations, which
are, at least theoretically, doable. However, as we mentioned, the
ideal $I_{\CC}$ is not known in general. The equivariant class
defined in Section \ref{equivariant_class_section} provides an
answer.

\begin{theorem} \label{coeff_GW}
The coefficient of $d_1^{q_1}d_2^{q_2}\ldots d_k^{q_k}$ in
$[Y_{\CC}]$ is $(-1)^{\codim \CC} N(\CC;q_1,q_2,\ldots,q_k)$.
\end{theorem}

The proof will be given in the next section.

\begin{example} \rm
The matrix matroid variety corresponding to the configuration of
Example \ref{steiner_construction} is a product of matrix Schubert
varieties. Hence, its equivariant class is computed by
Theorem~\ref{schur_tp} to be
$$(c_1-d_1-d_3-d_5)(c_1-d_2-d_3-d_4)(c_1-d_1-d_2-d_6).$$
The coefficient of $d_1d_2d_3$ is $-2$, reproducing the result of
Example \ref{steiner_construction}.
\end{example}

\begin{example}\label{Sch_problem}
\rm Special cases of $N(\CC;\q)$'s are solutions to
certain so-called {\em Schubert problems}. We illustrate this with the prototype of Schubert problems:
how many straight lines intersect 4 generic lines in $\PP^3$? In our language the answer is
$N(\CC_{(0,1,3,0,0)};1,1,1,1)$. According to Theorems~\ref{coeff_GW} and \ref{schur_tp} we have
$N(\CC_{(0,1,3,0,0)};1,1,1,1)=$ coefficient of $d_1d_2d_3d_4$ in
$$\det\begin{pmatrix}
d_1d_2+\ldots+d_4d_5 & d_1d_2d_3+\ldots+d_3d_4d_5 \\
d_1+\ldots+d_4 & d_1d_2+\ldots+d_3d_4
\end{pmatrix},$$
which is clearly $\binom{4}{2}-\binom{4}{1}=2$.
\end{example}

\begin{remark} Certain $N(\CC;\q)$ invariants are 0 for obvious reasons. For example, if
$P_1$, $P_2$, and $P_3$ are on one line in the configuration $\CC$, then $N(\CC;0,0,0,q_4,\ldots,q_k)=0$.
Indeed, $P_1$, $P_2$, and $P_3$, being on one line, can not be three {\em generic} points.
Similarly, if there is a subset $I=\{i_1,\ldots,i_s\}\subset [k]$ such that $\sum (q_{i_j}+1)>r_{\CC}(\{i_1,\ldots,i_s\})$, then
obviously $N(\CC;\q)=0$. It is easy to see that the existence of such an $I$ is the only reason for vanishing $N(\CC;\q)$.
\end{remark}

The method in Section \ref{calculation} to calculate $[Y_{\CC}]$ for the configurations $\CC_M$ (Menelaus), $\CC_C$ (Ceva), $\CC_P$ (Pappus), and $\CC_D$ (Desargues) (just like any other configuration we tried) works, leading to the knowledge of all the Gromov-Witten invariants of these configurations. Below we present some information on these invariants.

\begin{description}
\item[Menelaus] {\em All} non-zero coefficients of
the pure $d$ monomials of $[Y_{\CC_M}]$ are 1---for example the one studied in Example \ref{mene_enum_example}.
\item[Ceva] The same holds for the Ceva configuration: the range of invariants is $\{0,1\}$.
\item[Pappus] The range of the $N(\CC_P;\q)$ invariants is $\{0,1,2,3,4,5\}$. Here are some sample results.
\begin{itemize}
\item{} $N(\CC_P;1,1,1,1,1,1,1,0)=5$, that is, the number of Pappus configurations on the plane whose $i$'th vertex is on a pre-described generic line $l_i$ for $i=1,\ldots,8$, and whose 9'th point is a pre-described generic point, is 5. We know no other way of finding this number.
\item{} $N(\CC_P; 2,0,0,1,1,1,1,1,1)=4$. The argument in Example \ref{steiner_construction} would suggest the wrong answer 5.
\item{} $N(\CC_P; 1,1,0,2,1,1,0,0,2)=N(\CC_P;1,1,1,1,1,0,0,1,2)=3$.
\end{itemize}
\item[Desargues] Again, the range of the invariants $N(\CC_D;\q)$ is $\{0,1\}$.
\end{description}

It would be interesting to find a geometric interpretation of the property of a configuration, which is equivalent to the condition that the range of $N(\CC;\q)$ is $\{0,1\}$.

\section{Proof of Theorem \ref{coeff_GW}} \label{section:proof}

We are going to present a proof of Theorem \ref{coeff_GW} which also
proves positivity and enumerative properties of other coefficients.

Let $\CC$ be a configuration of $k$ vectors in $\C^n$, such that $Y_{\CC}$ is a codimension $l$ subvariety of $\Hom(\C^k,\C^n)$.

Let $\tau_n$ be the universal tautological bundle over the Grassmannian $\Gr_n\C^{\infty}$ (universal subbundle) which we will approximate with the finite Grassmannian $\Gr_n\C^N$ for a large $N>>n,k$. Below we will refer to certain number as `large'; by this we mean that those numbers tend to infinity as $N\to \infty$. The cohomology of the finite Grassmannian is a factor of the cohomology $H^*(\Gr_n\C^{\infty})=\Z[c_1,\ldots,c_n,d_1,\ldots,d_k]$ by an ideal with large degree generators. In our notations we will ignore this ideal, and identify the two cohomology rings.

Let $\phi:\Gr_n\C^N\to\Gr_n\C^N$ be the (necessarily non-holomorphic) classifying map of the dual vector bundle $\tau_n^*$, and consider the induced diagram
\[
\xymatrix{
\widehat{\Sigma}_{\CC} \ar@{^{(}->}[r] & \Hom(\tau_1^k,\tau_n^*) \ar[r]^{\psi} \ar[d]_\xi & \Hom(\tau_1^k,\tau_n) \ar[d] & \Sigma_{\CC} \ar@{_{(}->}[l] \\
 & \Gr_n\C^N \times (\PP^{N-1})^k \ar[r]^{\phi\times id} & \Gr_n\C^N \times (\PP^{N-1})^k
}
\]
Here $\psi$ is the map induced by $\phi\times id$, and $\Sigma_{\CC}$ is the collection of the copies of $Y_{\CC}$ in each fiber of the bundle $\Hom(\tau_1^k,\tau_n)$. That is, by definition,
the cohomology class represented by $\Sigma_{\CC}$ in the cohomology $H^*(\Hom(\tau_1^k,\tau_n))=H^*(\Gr_n\C^N \times (\PP^{N-1})^k)$ is the equivariant class $[Y_{\CC}]$. We set $\widehat{\Sigma}_{\CC}=\psi^{-1}(\Sigma_{\CC})$.

The homomorphism $(\phi\times id)^*:\Z[c_1,\ldots,c_n,d_1,\ldots,d_k]\to\Z[c_1,\ldots,c_n,d_1,\ldots,d_k]$  induced by the map $\phi\times id$ maps $c_i$ to $(-1)^ic_i$, and $d_i$ to $d_i$. Hence we have that
$$[\widehat{\Sigma}_{\CC} \subset \Hom(\tau_1^k,\tau_n^*)]=(\phi \times id)^*[{\Sigma}_{\CC} \subset \Hom(\tau_1^k,\tau_n)]=[Y_{\CC}]|_{c_i\mapsto (-1)^i c_i}.$$

Observe that the bundle $\xi$ has a large dimensional space of sections. Indeed, let $\alpha:\C^N \to (\C^N)^*$ be a linear map. Then for the
map
$$s_\alpha: (L^n\leq \C^N, l^1_1,\ldots,l^1_k\leq \C^N) \mapsto \left( v  \mapsto \left( \alpha(v)|_{l_1},\ldots,\alpha(v)|_{l_k}\right)\right)$$
we have that $s_\alpha^*$ is a section of $\xi$.
Therefore, using the Kleiman-Bertini transversality theorem, we can assume that $s$ is transversal to $Y_{\CC}$ for an appropriate choice of a section $s$ of $\xi$. Hence we have that $V_{\CC}=s^{-1}(\widehat{\Sigma}_{\CC})$ is a codimension $l$ subvariety of $\Gr_n\C^N \times (\PP^{N-1})^k$, which represents $[Y_{\CC}]|_{c_i\mapsto (-1)^i c_i}$ in the cohomology of $\Gr_n\C^N \times (\PP^{N-1})^k$.

Now fix a complete flag in $\C^N$, and consider the products of Schubert varieties
$$S_{\lambda,\q}=S_\lambda \times (S_{q_1}\times\ldots \times S_{q_k}) \subset \Gr_n\C^N \times (\PP^{N-1})^k.$$
Here $\lambda$ is a partition (with number of parts $\leq n$), and $S_i\subset \PP^{N-1}$ is a linear space of codimension $i$, with $\q=(q_1,\ldots,q_k)$. The cohomology classes represented by $S_{\lambda,\q}$'s form a basis of the cohomology group, and by the Giambelli formula of Schubert calculus we know that
$$[S_{\lambda,\q}]=\Delta_{\lambda}(-c_1,c_2,-c_3,c_4,\ldots,(-1)^nc_n)\cdot\prod_{i=1}^k(-d_i)^{q_i}.$$
Recall also that this basis is self dual in the sense that $\int [S_{\lambda,\q}][S_{\mu,\w}]=0$ unless $\mu=\lambda'$, the complement of $\lambda$ in the $n\times (N-n)$ rectangle, and $\w=\q'$, ie, $\q+\w=(N,\ldots,N)$ (in which case, the integral is 1).

By appropriate choice of the section $s$ we may also assume that $V_{\CC}$ is transversal to the  Schubert varieties $S_{\lambda,\q}$.
Let $|\lambda| + \sum q_i =\dim(\Gr_n\C^N \times (\PP^{N-1})^k) -l$. Then we have that
$$\#(V_{\CC} \cap S_{\lambda,\q})= \int [V][S_{\lambda,\q}],$$
which is then the coefficient of $[S_{\lambda',\q'}]$ when $[V_{\CC}]$ is written as a linear combination of $[S_{\mu,\w}]$'s. By rephrasing we obtain the following theorem.


\begin{theorem} \label{general_coeffs}
 If $[Y_{\CC}]$ is expressed as a linear combination
$$\sum_{\mu,\w} a_{\mu,w}\cdot \Delta_{\mu}(c_1,c_2,\ldots,c_n)\prod_{i=1}^k(-d_i)^{w_i},$$
then the coefficient $a_{\lambda,\q}$ is equal to the intersection number $\#(V_{\CC} \cap S_{\lambda',\q'})$. Hence all coefficients in this linear combination are nonnegative.
\end{theorem}

Let us now consider the special case of $\lambda$ being the empty partition, and $\sum q_i=l$. We obtain that the coefficient of $(-1)^l\prod_i d_i^{q_i}$ in $[Y_{\CC}]$ is the intersection number
$$\#(V \cap (point \times H_1 \times \ldots \times H_k)),$$
where $H_i$ are generic linear subspaces in $\PP^N$ of total codimension $l$.  Identifying the {\em point} in $\Gr_n\C^N$ with $\C^n$ this intersection number is the same as the definition of $N(\CC;\q)$. This proves Theorem \ref{coeff_GW}.

\smallskip

Theorem \ref{coeff_GW} shows that geometric interpretation of the pure $d$ coefficients of $[Y_{\CC}]$, but the other extreme, the pure $c$ coefficients are also noteworthy. By the definition of $GL(n)$-equivariant cohomology we have

\begin{theorem}
Suppose $\pi:E \to B$ is a rank $n$ vector bundle with Chern classes $c_1,\ldots,c_n$. Assume that $\pi$ has $k$ sections satisfying a certain transversality property. Then the cohomology class in $B$ represented by the subvariety
$$\{b\in B| s_1(b), \ldots, s_k(b) \text{ form a configuration belonging to } Y_{\CC}\subset \pi^{-1}(b) \}$$
is equal to $[Y_{\CC}]^*$.
\end{theorem}

The transversality property can be easily phrased. Over the real numbers it is a generic property of $k$-tuples of sections. Thus one obtains a result on the parity of (the cohomology class represented by) the points over which $k$ generic sections of a real projective space bundle form a given configuration $\CC$.

\begin{remark} \rm
Certain facts suggest some kind of relations between the $c$ and the $d$ variables of $[Y_{\CC}]$. One of these facts is that either one can be used to calculate the degree of $Y_{\CC}$---hence they can not be independent. More generally, it can be shown that $[Y_{\CC}]$ can be written as a polynomial of the weights $\gamma_i-d_j$ of the representation in Section \ref{equivariant_class_section}. Another fact is that for matrix Schubert varieties the pure $c$ part determines the whole $[Y_{\CC}]$ (up to permutation of indexes). Since the pure $c$ part of $[Y_{\CC}]$ can be presented in a more compact way in general, it would be interesting to see the relation in general. However, we may not expect that eg. the pure $c$ part determines the pure $d$ part for any $\CC$. For example, let $\CC_1$ be the configuration of $7$ points on the projective plane with the collinearities $123$, $145$, and $167$ (and otherwise general). Let $\CC_2$ be with the collinearities $123$, $345$, $567$ (and otherwise general). The pure $c$ part of the equivariant classes of both of these are $c_1^3$. The pure $d$ parts are essentially different (see Theorem~\ref{schur_tp}).
\end{remark}

\section{Stabilization}\label{dstab}

The interior structure of natural infinite sequences of equivariant cohomology classes of geometrically relevant varieties has remarkable
connections with the theory of symmetric functions \cite{nakajima_book}, and iterated residue identities for hyperplane arrangements \cite{bsz06, thom_series}. In this section we make the first step towards exploring similar relations for the classes $[Y_{\CC}]$, by showing the property analogous with the so-called d-stability property of Thom polynomials of contact singularities, cf. \cite[Sect. 7.3]{thom_series}. A byproduct---important in Section \ref{calculation}---is Theorem \ref{width} on the vanishing of certain coefficients of $[Y_{\CC}]$.

Let $\CC$ be a configuration of $k$ vectors in $\C^s$, spanning $\C^s$, and let $\codim (Y_{\CC}\subset \C^{ s\times k})=l$.
For $n\geq s$ let $\CC^{\#n}$ be obtained from $\CC$ through the natural embedding of $\C^s$ into $\C^n$. Hence $\CC^{\#n}$ is a configuration of $k$ points in $\C^{n}$, spanning an $s$ dimensional subspace. It is easy to check that $\codim (Y_{\CC^{\#n}}\subset \C^{n\times k})=l+(k-s)(n-s)$.
Our goal is to compare the classes
\begin{equation}\label{2classes}[Y_{\CC}]\in \Z[c_1,\ldots,c_s,d_1,\ldots,d_k]\qquad\hbox{and}\qquad [Y_{\CC^{\#n}}]\in \Z[c_1,\ldots,c_{n},d_1,\ldots,d_k].\end{equation}
The relation between the two classes must involve nontrivial algebra, since it involves nontrivial geometry---consider for example the case when $\CC$ is the collection of 4 generic vectors in $\C^2$. In this case $[Y_{\CC}]=1$ while $[Y_{\CC^{\#4}}]$ has a term $2d_1d_2d_3d_4$, whose coefficient is the solution of the Schubert problem of Example \ref{Sch_problem}.

Below we will work with the Chern roots $\gamma_i$ of $GL(m)$ ($m=s$ or $n$), instead of the Chern classes $c_i$. That is, we identify $H^*(BGL(m))=\Z[c_1,\ldots,c_m]$ with the symmetric polynomials of $\gamma_1,\ldots,\gamma_m$; where the $i$'th elementary symmetric polynomial is $c_i$. In our notation, hence, $Y_{\CC}$ may be a polynomial in $c_i$'s and $d_i$'s, or a polynomial in $\gamma_i$'s and $d_i$'s necessarily symmetric in the $\gamma_i$'s. If $S$ is an $s$-element subset of $[n]$, then $[Y_{\CC}](\gamma_S)$ will denote the the value of $[Y_{\CC}]$ if we substitute the variables $\gamma_i$, $i\in S$ for $\gamma_1,\ldots,\gamma_s$.

\begin{theorem} \label{lok}
 Let $\binom{[n]}{s}$ denote the set of $s$-element subsets of $[n]$. For such a subset $S$, let $\bar{S}$ denote the complement of $S$ in $[n]$. Then we have
\begin{equation}
[Y_{\CC^{\#n}}]=\sum_{S\in \binom{[n]}{s}} \frac{[Y_{\CC}](\gamma_S) \cdot \prod_{i\in \bar{S}} \prod_{j=1}^k (\gamma_i-d_j)}
{\prod_{i\in \bar{S}} \prod_{j\in S} (\gamma_i-\gamma_j)}.
\end{equation}
\end{theorem}

\begin{proof}Let $\varepsilon^k$ denote the trivial bundle of rank $k$, and let $\tau_s$ be the tautological bundle (or rank $s$) over $Gr_s\C^n$.
The embedding of bundles $\tau_s\subset \varepsilon^n$ induces the embedding of bundles $i:\Hom(\varepsilon^k,\tau_s)\to \Hom(\varepsilon^k,\varepsilon^n)$. The maximal torus $U(1)^n\times U(1)^k$ of $G_{n,k}$ acts on the following diagram
\[
\xymatrix{
Y_{\CC}(\xi)\subset& \Hom(\varepsilon^k,\tau_s)  \ar@{^{(}->}@<-3pt>[rr]^{i}  \ar[drr]_{\xi} &  &  \Hom(\varepsilon^k,\varepsilon^n)\ar[d]^{\pi_1}\ar[r]^{\pi_2} & \Hom(\C^k,\C^n) & \supset Y_{\CC^{\#n}} \\
& & & Gr_s\C^n,  }
\]
where $\pi_1$ are $\pi_2$ the projections of $\Hom(\varepsilon^k,\varepsilon^n)=Gr_s\C^n\times \Hom(\C^k,\C^n)$, and $Y_{\CC}(\xi)$ is the collection of the $Y_{\CC}$-points in each fiber of $\xi$. The composition $\pi_2\circ i$ is a birational map from $Y_{\CC}(\xi)$ to $Y_{\CC^{\#n}}$. Therefore we can apply the fibered version of the Atiyah-Bott localization theorem, Theorem (3.8) in \cite{bsz06} (see also \cite[Prop.5.1]{thom_series}), and we obtain the theorem.
\end{proof}

Another relation between the classes (\ref{2classes}) stems from the following theorem.

\begin{theorem} \label{dstabCC} Let $n>s$, and $[Y_{\CC^{\#n}}]= \sum \gamma_{n}^i p_i(\gamma_1,\ldots,\gamma_{n-1},d_1,\ldots,d_k)$. Then
\begin{enumerate}
\item \label{egy} $p_i=0$ for $i>k-s$, and
\item \label{ketto} $p_{k-s}=k_{\CC}\cdot[Y_{\CC^{\#(n-1)}}]$ for an integer $k_{\CC}$.
\end{enumerate}
\end{theorem}

\begin{proof} Apply Theorem 2.1 of \cite{dstab}. \end{proof}

To explore the algebraic consequences of Theorem \ref{dstabCC} we define lowering and raising operators on constant width polynomials.

\begin{definition} The width of a monomial in $\Z[c_1,\ldots,c_n]$ is the number of factors in it. The width of a polynomial is the width of its widest term. Let $P^n_{w}$ be the vector space of width $\leq w$ polynomials in $\Z[c_1,\ldots,c_n]$. The lowering operator $L^n_{w}: P^{n+1}_{w} \to P^{n}_{w}$ is defined to be the linear extension of
$$L^n_{w}\left(c_{i_1}c_{i_2}\ldots c_{i_w} \right) = c_{i_1-1}c_{i_2-1}\ldots c_{i_w-1},$$
and $L^n_{w}(c_I)=0$ if the width of $c_I$ is less than $w$ ($c_0$ is defined to be 1). $L^n_w$ decreases the degree by $w$.
The raising operator (increasing the degree by $w$) $R^n_{w}: P^n_{w} \to P^{n+1}_{w}$ is defined by
$$p=\sum_\lambda a_\lambda \Delta_\lambda(c_1,\ldots,c_n) \qquad \mapsto \qquad \sum_\lambda a_\lambda \Delta_{(\lambda_1+1,\ldots,\lambda_{w}+1)}(c_1,\ldots,c_{n+1}),$$
where $\sum_\lambda a_\lambda \Delta_\lambda$ is the unique expression of $p$ as a linear combination of Schur polynomials corresponding to partitions with $w$ parts.
\end{definition}

For instance we have $L^2_{3}(c_1c_2c_3+5c_3^2)=c_1c_2$, and
$R^2_{3}(c_1c_2)=R^2_{3}(\Delta_{210}(c_1,c_2))=\Delta_{321}(c_1,c_2,c_3)=$ $c_1c_2c_3-c_3^2.$
We have the one-sided inverse property $L^n_{w}\circ R^n_{w}=$id, but not the other way around.

\begin{theorem} \label{width}
Let $\CC$ be a configuration of $k$ vectors in $\C^n$, spanning an $s$ dimensional subspace.
\begin{itemize}
\item The width of $[Y_{\CC}]^*$ is at most $k-s$.
\item If $[Y_{\CC}]^*$ is written in the Schur basis $\Delta_{\lambda}$, then all occurring $\lambda$ have at most $k-s$ parts.
\end{itemize}
\end{theorem}

\begin{proof} If $[Y_{\CC}]^*$ had a term of width $i>k-s$, then $L^n_i([Y_{\CC}]^*)$ would not be 0, contradicting to
Theorem \ref{dstabCC} (\ref{egy}) (cf. \cite[Cor.2.5]{dstab}). This proves the first statement. The second is a combinatorial rephrasing of the first one.
\end{proof}

\begin{theorem} \label{combinat_lemma}
Let $n,k\geq s$, and let us use the notations of Theorem \ref{lok}. Let $\lambda$ be a partition with at most $k-s$ parts. Then we have
\begin{equation}
R^{n-1}_{k-s}\circ \ldots \circ R^{s+1}_{k-s} \circ R^{s}_{k-s} \left( \Delta_{\lambda}(c_1,\ldots,c_s)\right)=
\sum_{S\in \binom{[n]}{s}} \frac{\Delta_\lambda(\gamma_S) \cdot \prod_{i\in \bar{S}} \gamma_i^k}   
{\prod_{i\in \bar{S}} \prod_{j\in S} (\gamma_i-\gamma_j)}.
\end{equation}
\end{theorem}

\begin{proof} The polynomial $\Delta_\lambda(c_1,\ldots,c_s)$ is the equivariant class $[Y_{\CC}]^*$ of an appropriate matrix Schubert variety, according to Theorem \ref{schur_tp}. Then Theorem \ref{lok} gives that the right hand side is the $GL(n)$ equivariant class of another matrix Schubert variety. Checking the indexes, and applying Theorem \ref{schur_tp} again we obtain the left hand side.
\end{proof}

It would be interesting to find a combinatorial proof of this theorem, possibly along the line of the multivariate Lagrange interpolation formula for
symmetric functions \cite{chen_louck}.

Finally,  we have the description of the relation between the pure $c$ parts of the equivariant classes (\ref{2classes}).

\begin{theorem}
Let $\CC$ be a configuration of $k$ vectors in $\C^s$, spanning $\C^s$. 
Let $n\geq s$, and let $\CC^{\#n}$ be obtained from $\CC$ by the natural embedding $\C^s\subset \C^n$. Then $[Y_{\CC^{\#n}}]^*$ has width at most $k-s$, and
$$[Y_{\CC^{\#n}}]^*=R^{n-1}_{k-s}\circ\ldots\circ R^{s+1}_{k-s} \circ R^s_{k-s}\left( [Y_{\CC}]^* \right).$$
\end{theorem}

\begin{proof} According to Theorem \ref{width} we can express $[Y_{\CC}]^*$ as a linear combination of $\Delta_\lambda$ polynomials with each $\lambda$ having at most $k-s$ parts. Let us apply the operation
$p\mapsto \sum_{S} \frac{ p \cdot \prod_{i\in \bar{S}} \gamma_i^k    }
{\prod_{i\in \bar{S}} \prod_{j\in S} (\gamma_i-\gamma_j)}$
to this expression. Theorem \ref{combinat_lemma} gives our result.
\end{proof}

In particular the constant $k_{\CC}$ above is 1. Furthermore, the pure $c$ part of $[Y_{\CC}]$ determines $[Y_{\CC^{\#n}}]$ by adding a $(k-s)\times (n-s)$ rectangle to each $\lambda$ in its Schur expansion. The analogous phenomenon for equivariant classes of contact singularities is called the ``finiteness of Thom series'', see \cite{thom_series}. However, finiteness of Thom series seems more the exception than the rule for contact singularities. The only known finite Thom series correspond to a trivial case (the algebras $\Z[x_1,\ldots,x_n]/(x_1,\ldots,x_n)^2$), the Giambelli-Thom-Porteous formula.

Denoting the Menelaus configuration considered in a subspace of $\C^n$ ($n\geq 3$) by $\CC_M^{\#n}$ we obtain that
$$[Y_{\CC_M^{\#n}}]^*=3\Delta_{n-1,n-2,n-2}+2\Delta_{n-1,n-1,n-3}+ 3\Delta_{n,n-2,n-3}.$$

\begin{remark}\rm The only property of the representation of Section \ref{equivariant_class_section} used in this section was that this representation is a quiver representation. Hence the suitable rephrasing of the localization, width, and vanishing results above are valid for all quiver representations.
\end{remark}

\section{The calculation of equivariant classes of matrix matroid
varieties}\label{calculation}

The standard straightforward methods to calculate equivariant classes of invariant subvari\-e\-ties---such as the method of resolution or
(Gr\"obner) degeneration---assume more knowledge on the ideal of the variety than we have about the ideal of matrix matroid varieties.

What we can do is listing certain properties of the class $[Y_{\CC}]$, and hope that a computer search proves that there is only one element of the polynomial ring $\Z[c_1,\ldots,c_n,d_1,\ldots,d_k]$ that satisfies all these properties. The main such property---which we will call Interpolation property---is motivated by methods used in the theory of Thom polynomials of singularities.

\subsection{Interpolation}

For a configuration $\DD\in \C^{n\times k}$ let $G_{\DD}$ denote its stabilizer subgroup in $G_{n,k}$. The embedding $G_{\DD} \to G_{n,k}$ induces a map between classifying spaces $BG_{\DD} \to BG_{n,k}$, and, in turn, a homomorphism between rings of characteristic classes $\phi_{\DD}: H^*(BG_{n,k}) \to H^*(BG_{\DD})$.

\begin{theorem}\label{restriction} \cite[Th.3.2]{cr} If $\DD\not\in Y_{\CC}$ then $\phi_{\DD} ( [Y_{\CC}]) =0$. \end{theorem}

Theorem \ref{restriction} is a homogeneous interpolating condition on $[Y_{\CC}]$. To obtain a non-trivial condition, however, we need to find a configuration $\DD$, outside of $Y_{\CC}$, with reasonably large symmetry group. Let us illustrate the usage of this theorem with an example.

\begin{example} \label{menelaus_comp} The calculation of $[Y_{\CC_M}]$. \rm
Consider the following configuration $\DD_{1|2|6}$: $v_1$, $v_2$, and $v_6$ are three generic vectors in $\C^3$, while $v_3=v_4=v_5=0\in \C^3$.
Clearly $\DD_{1|2|6}$ is not in the closure of $X_{\CC_M}$, since for all configurations in $X_{\CC_M}$ the vectors $v_1, v_2$, and $v_6$ are coplanar.
Therefore $\phi_{\DD_{1|2|6}}([Y_{\CC_M}])=0$. The stabilizer subgroup of $\DD_{1|2|6}$ is $U(1)^6$ with the embedding into $G_{3,6}$ via
$(\diag(\alpha,\beta,\gamma),\diag(\alpha,\beta,\delta,\eta,
\theta,\gamma))$. Hence---by abusing language and identifying general elements of a $U(1)$ with the first Chern class of $U(1)$---the map $\phi_{\DD_{1|2|6}}:\Z[c_1,c_1,c_3,d_1,\ldots,d_6] \to \Z[\alpha,\beta,\gamma,\delta,\eta,\theta]$ maps
$$c_1\mapsto \alpha+\beta+\gamma,\qquad c_2\mapsto \alpha\beta+\alpha\gamma+\beta\gamma,\qquad c_3\mapsto \alpha\beta\gamma,$$
$$d_1\mapsto \alpha,\qquad d_2\mapsto \beta,\qquad d_3\mapsto \delta,\qquad d_4\mapsto\eta, \qquad d_5\mapsto \theta,\qquad d_6\mapsto \gamma.$$
The vanishing of $[Y_{\CC_M}]$ at this map is a non-trivial interpolation property of $[Y_{\CC_M}]$. In fact, one finds that the degree 4 part of the intersection of the kernels of the $\phi_{\DD_{1|2|6}}$, $\phi_{\DD_{1|3|5}}$, $\phi_{\DD_{2|3|4}}$, and $\phi_{\DD_{4|5|6}}$ is 2-dimensional. (This is not surprising in the light of the decomposition (\ref{menelaus_decomposition}).) Now let $\DD_{124|356}$ be the following configuration: $v_1=v_2=v_4$ and $v_3=v_5=v_6$ are two different nonzero vectors in $\C^3$. This configuration is not in the closure of $X_{\CC_M}$ because of Menelaus' theorem. Indeed, the left hand side of (\ref{menelaus identity}) is $-1$ for any configuration in $X_{\CC_M}$, but it is $\infty$ for $\DD_{124|356}$. As a consequence,
$[Y_{\CC_M}]$ must vanish at the map $\phi_{\DD_{124|356}}:\Z[c_1,c_1,c_3,d_1,\ldots,d_6] \to \Z[\alpha,\beta,\gamma]$,
$$c_1\mapsto \alpha+\beta+\gamma,\qquad c_2\mapsto \alpha\beta+\alpha\gamma+\beta\gamma,\qquad c_3\mapsto \alpha\beta\gamma,$$
$$d_1\mapsto \alpha,\qquad d_2\mapsto \alpha,\qquad d_3\mapsto \beta,\qquad d_4\mapsto\alpha, \qquad d_5\mapsto \beta,\qquad d_6\mapsto \beta.$$
It turns out that there is only a 1-dimensional set of
degree 4 polynomials in $\Z[c_1,c_1,c_3,d_1,\ldots,d_6]$ vanishing at the five maps $\phi_{\DD_{1|2|6}}$, $\phi_{\DD_{1|3|5}}$, $\phi_{\DD_{2|3|4}}$, $\phi_{\DD_{4|5|6}}$, and $\phi_{\DD_{124|356}}$. Normalization is achieved, for example, by observing that the coefficient of $d_3^2d_6^2$ in $[Y_{\CC_M}]$ has to be 1, due to Theorem \ref{coeff_GW}.
\end{example}

Observe that the application of Theorem \ref{restriction} to calculate $[Y_{\CC}]$ rhymes to the method of determining the ideal of $Y_{\CC}$ discussed at the end of Section \ref{examples:ideal}. That is, we first deal with the trivial conditions following from the closed conditions on $X_{\CC}$, then need to work with some extra equations besides these naive ones. What makes the equivariant cohomology calculation easier is that here we do not have to have a full understanding of {\em all} the fake components, or {\em all} the extra geometry theorems of the configuration. It is enough to find {\em some} of these, use these to find an interpolations constraint. And it is clear when we can stop: as soon as we find enough interpolation constraints to cut down the dimension of the solution set to 1, we can be sure we found $[Y_{\CC}]$.

\subsection{Calculation in practice}\label{calprac}

The three main conditions we may use to calculate the equivariant class $[Y_{\CC}]$ are
\begin{itemize}
\item the interpolation conditions, Theorem \ref{restriction};
\item the enumerative conditions, Theorem \ref{coeff_GW};
\item the width condition, Theorem \ref{width}.
\end{itemize}
The first one depends on the choice of the test configuration $\DD$. The second one depends on the choice of the numbers $\q$. For certain choices of $\DD$ and $\q$ these conditions are far from being straightforward, because we do not know whether $\DD$ belongs to $Y_{\CC}$, or the number $N(\CC;\q)$. For some other choices, however, simple arguments answer these questions, and hence we have explicit constraints of $[Y_{\CC}]$. See, for example, the calculation of $[Y_{\CC_M}]$ above.

It is quite possible that the interpolation conditions themselves are enough to determine the equivariant class $[Y_{\CC}]$ up to a scalar. For some other representations the analogous statement is a theorem, eg. \cite[Thm. 3.5]{cr}. However the proof there depends on a condition of the representation (called Euler condition in \cite[Def. 3.3]{cr}, closely related to the ``equivariantly perfect'' condition of \cite[Sect.1.]{ab}). This condition does not hold for the representation of Section~\ref{equivariant_class_section}.

What works in practice, is the combination of the three constraints. For all the configurations the authors considered (much more than the ones presented in this paper) there is only one polynomial of degree $\codim Y_{\CC}$ in $\Z[c_1,\ldots,c_n,d_1,\ldots,d_k]$ satisfying the simple straightforward constraints obtained from interpolation and enumeration, together with the width condition. We conjecture this holds for all configurations.

\section{Hierarchy}\label{hierarchy}

The interpolation method highlights the importance of the hierarchy of the sets $Y_{\CC}$. In fact, the effective usage of the interpolation method to calculate $[Y_{\CC}]$ assumes that we have another configuration $\DD$ such that $\DD \not\in Y_{\CC}$. To indicate the non-triviality of this problem we challenge the reader with the problem of deciding whether the configuration $\DD_{134|256}$ is contained in $Y_{\CC_C}$ (for notations see Figure \ref{menelaus ceva_figure} and Example \ref{menelaus_comp}).

The hierarchy of the sets $X_{\CC}$ is not ``normal'', in the sense that there are examples of configurations $\CC$ and $\DD$ such that certain
points of $X_{\DD}$ are in the closure of $X_{\CC}$, some others are not. A small example is $\CC=\CC_M$, $\DD=$6 points on one line. Hence, we restrict our attention to the case when $\DD$ is an orbit of the action in Section \ref{equivariant_class_section}. In this case Theorem \ref{restriction} yields the following: if $\phi_{\DD}([Y_{\CC}])\not=0$ then $Y_{\DD}\subset Y_{\CC}$.

The vanishing of $\phi_{\DD}([Y_{\CC}])$ has no chance of determining the adjacency of $Y_{\CC}$ and $Y_{\DD}$ if the stabilizer group $G_{\DD}$ is trivial (ie. it is the kernel of the representation, $U(1)$). However, when $G_{\DD}$ is larger, we have found no counterexample to the following conjecture.

\begin{conjecture}
Let $\CC$ and $\DD$ be configurations of $k$ points in $\C^n$. Suppose that projectivizations of the non-zero vectors in $\DD$ form a projectively independent set (hence of cardinality $\leq n$). Then
$$Y_{\DD} \subset Y_{\CC}\qquad  \Leftrightarrow \qquad \phi_{\DD}([Y_{\CC}])\not=0.$$
\end{conjecture}

Together with the effective algorithm of Section \ref{calprac} computing $[Y_{\CC}]$, this conjecture would serve as a computable criterion of
hierarchy, cf. \cite{zsolt}.

\bibliography{matroid}
\bibliographystyle{alpha}

\end{document}